\documentclass[11pt]{article}
\usepackage{hyperref}
\usepackage{amssymb,amsmath,amsthm}
\usepackage{palatino}

\newtheorem{thm}{Theorem}[section]

\newtheorem{lemma}{Lemma}[section]

\newtheorem{claim}{Claim}[section]

\newtheorem{cor}{Corollary}[section]

\newtheorem{define}{Definition}[section]

\renewenvironment{proof}[1][]{\begin{trivlist}
\item[\hspace{\labelsep}{\bf\noindent Proof#1:\/}] }{\qed\end{trivlist}}

\newcommand{\eps}[0]{\varepsilon}

\newcommand{\Ex}[0]{ \mathbb{E} }

\newcommand{\set}[2]{ \{ #1 , \ldots #2 \} }

\newcommand{\ud}{\,\mathrm{d}}

\newcommand{\enote}[1]{}

\newcommand{\remove}[1]{}

\title{The property of having a $k$-regular subgraph has a sharp threshold }
\date{28 March, 2012}
\author {
Shoham Letzter
\thanks{Institute of Mathematics, Hebrew University, Jerusalem, Israel. email: shoham.letzter at mail.huji.ac.il.}
}
\begin{document}

{
\maketitle
}

\begin{abstract}
We prove that the property of containing a $k$-regular subgraph in the random graph model $G(n,p)$ has a sharp threshold for $k\ge3$. We also show how to use similar methods to obtain an easy prove for the (known fact of) sharpness of having a non empty $k$-core for $k\ge3$.
\end{abstract}

\section{Introduction}
Let $G(n,p)$ denote the random graph with vertex set $\set{1}{n}$, in which every edge appears  independently with probability $p=p(n)$.
For a fixed $n$, a \emph{graph property} is a subset of the set of all graphs on $\set{1}{n}$.
A graph property $\cal{A}$ \enote{Change all $A$ to $\cal{A}$} is called \emph{monotone} if \enote{ it is preserved under
adding edges} it is preserved under addition of edges.
A graph property is called \emph{symmetric} if it is closed under 
permutations of the vertices, i.e. does not depend on the vertex labelling.
We shall consider sequences of graph properties $\mathcal{A}=\mathcal{A}(n)$, where, typically,
the properties have something in common (e.g. the property, or more precisely the sequence of properties,
of connected graphs).

Let
\begin{equation*}
\mu(p)=\mu(\mathcal{A},n,p)=\Pr[G(n,p)\in \mathcal{A}].
\end{equation*}
Note that for monotone $\mathcal{A}$, $\mu(p)$ is monotone on $[0,1]$.
A function $p^*=p^*(n)$ is called a \emph{threshold function} for $\cal{A}$ if for every $p=p(n)$
\begin{equation*}
\lim_{n \to \infty}\mu(p)=\left\{
\begin{array}{rl}
1 & p\gg p^* \\
0 & p\ll p^*
\end{array} \right .
\end{equation*}
As was observed in \cite{threshold}, every monotone property has a threshold function.

A property $\cal{A}$ is said to have a \emph{sharp threshold} if there exists some $p^*=p^*(n)$
such that for every $\eps>0$ and every $p=p(n)$
\begin{equation*}
\lim_{n \to \infty}\mu(p)=\left\{
\begin{array}{rl}
1 & p> (1+\eps)p^* \\
0 & p< (1-\eps)p^*
\end{array} \right .
\end{equation*}
If a property does not have a sharp threshold we say that it has a \emph{coarse threshold}.
Note that if a property has a coarse threshold,  there exist $p=p(n)$, $C,\alpha>0$ such that $\alpha<\mu(p)<1-\alpha$ and
for infinitely many values of $n\in\mathbb{N}$, $p\frac{\ud\mu}{\ud p}|_{p=p}\le C$.

Let us look at two examples for monotone and symmetric graph properties. 
First, consider the property of being connected. Relatively straightforward analysis shows that this property has a sharp threshold with a threshold function $p(n)=\frac{\mathrm{log}n}{n}$. 
Consider also the example of containing a triangle. It can be shown that its threshold function is $\frac{1}{n}$ and that for every $c>0$, $G(n,\frac{c}{n})$ contains a triangle with probability tending to $1-e^{-c^3/6}$ when $n$ goes to infinity. Hence the threshold in this case is coarse.
The studying of the above examples and many others, 
may lead to the belief that properties with a coarse threshold are either the properties of 
containing a subgraph from a certain fixed list of graphs -  
properties which are called \emph{local} properties, or can be approximated by a local property. 
Roughly speaking, this is indeed the case, as was shown by Friedgut \cite{sharp-threshold} 
the approximability by a local property is a necessary condition 
for a monotone and symmetric graph property to have a coarse threshold. 
This theorem and variations on it, such as Bourgain's theorem \cite{Bourgain}, are to date the main instruments
in proving sharpness of thresholds.

In \cite{hunting}, this approach is streamlined and adapted to applications, however the version there
turned out to be slightly too weak for this article's purpose. 
Instead, the technique which was eventually used here is quite similar to the one appearing in \cite{k-color}.

Before stating the theorem which will be used here, let us introduce some notation.
Let $H$ be a fixed graph. Denote by $\Pr[\mathcal{A},n,p \vert H]=\Pr[\mathcal{A} \vert H]$  
the probability that $G(n,p)$ has property $\cal{A}$, conditional on the appearance of a
 \emph{specific} copy of $H$ in $G(n,p)$, i.e. some specific labelled graph which is isomorphic to $H$
 and whose vertices form a subset of $\set{1}{n}$. 
 Note that for a symmetric property $\cal{A}$, 
 this probability doesn't depend on the choice of the copy of $H$. 
 Also, although it might not be immediately obvious, it is the case 
 that conditioning on the appearance of a specific copy of $H$ 
 is different than conditioning on the appearance of some copy of $H$.

The following theorem gives a necessary condition for a property to have a coarse threshold. 
This theorem states, roughly, the idea that was presented earlier,
 that if some monotone and symmetric property has a coarse threshold, 
 it can be approximated by a local property. 
 It is a slight variation of a theorem appearing in \cite{sharp-threshold}.
\begin{thm}
Let $0<\alpha<1/2$, $C>0$.
For every $0<\eps<1$, there exist $B\in\mathbb{N}$ and $0<\beta<1$, such that for every $n\in\mathbb{N}$, every monotone and symmetric property $\cal{A}$ and every $0<p^*<1$ such that $\alpha<\mu(p^*)<1-\alpha$ and $p^*\cdot\frac{\ud \mu}{\ud p}\vert_{p=p^*} \leq C$, there exists a graph $H$ with the following properties:
\begin{itemize}
\item
$H$ has no more than $B$ edges.
\item
The probability that some copy of $H$ appears in $G(n,p)$ is at least $\beta$.
\item
$\Pr[\mathcal{A} \vert H] > 1 - \eps$.
\end{itemize}
\end{thm}
Clearly, the graph $H$ in the above theorem can be assumed to have no isolated vertices, and hence there are
only a finite number of possible isomorphism types of such graphs with at most $B$ edges. It follows that whenever 
a monotone symmetric graph property has a coarse threshold there is a single ``culprit" $H$ that can be blamed for
the situation for infinitely many values of $n$. Formally 
\begin{cor}
\label{coarse}
Let $\cal{A}$ be a monotone and symmetric graph property. If $\cal{A}$ has a coarse threshold, then there exist $\alpha>0$ and $p=p(n)$ such that $\alpha<\mu(p)<1-\alpha$, and for every $\eps>0$ there exist $\beta>0$ and a fixed graph $H$ such that for infinitely many $n\in\mathbb{N}$:
\begin{itemize}
\item
The probability that some copy of $H$ appears in $G(n,p)$ is at least $\beta$.
\item
$\Pr[\mathcal{A} \vert H] > 1 - \eps$.
\end{itemize}
\end{cor}

The main goal of this article is to show that the property of containing a $k$-regular subgraph, for a fixed $k\ge3$, has a sharp threshold, using the above theorem.
\begin{thm}
\label{regular}
Let $k\ge3$, and let $\cal{A}$ be the property of containing a $k$-regular subgraph. Then $\cal{A}$ has a sharp threshold.
\end{thm}
For $k=2$ it is known that the property of having a $2$-regular subgraph has a coarse threshold.
In fact, it has a semi sharp threshold. The probability that $G(n,\frac{c}{n})$ contains a cycle is bounded away from both $0$ and $1$ for every $0<c<1$, and tends to $1$ when $n$ goes to infinity, for every $c>1$.
For $k\ge3$ it is known that the threshold function $p=p(n)$ satisfies $p=\Theta(\frac{1}{n})$.
Moreover, in \cite{k-regular} it is shown that for every $k\ge 3$ there exist constants $d_k<D_k$ which depend linearly on $k$, such that for every $p=p(n)$
\begin{equation*}
\lim_{n \to \infty}\mu(p)=\left\{
\begin{array}{rl}
1 & p> \frac{D_k}{n} \\
0 & p< \frac{d_k}{n}
\end{array} \right.
\end{equation*}
There it is also conjectured that this property has a sharp threshold.

We shall also show that the method used to derive the sharpness of the threshold of 
containing a $k$-regular subgraph, can be used easily to derive sharpness of the property of containing 
a non empty $k$-core, for $k\ge 3$.
The \emph{$k$-core} of a graph, for some $k\ge2$, is its maximal subgraph in which all vertices have degree at least $k$.
Note that for $k=2$, containing a $2$-regular subgraph is equivalent to having a non empty $2$-core. Hence,  the roperty of having a non empty $2$-core has a coarse threshold.

The sharpness of the threshold for the emergence of a non trivial $k$-core has already been proved in \cite{k-core}, 
where the threshold function is computed precisely
 for every $k\ge 3$. 
 It is shown that the threshold function is $\frac{\lambda_k}{n}$ for a constant $\lambda_k>0$ which is computed in the paper.
However, the calculations describing the precise threshold behaviour are rather complex.
Therefore it is interesting to use our approach to derive the sharpness of the threshold for containing a 
non empty $k$-core, 
(though not computing the threshold function at all).
Furthermore, knowing that the threshold is sharp may, in principle, 
make finding the actual threshold function easier. whenever the probability of containing a non empty $k$-core
is bounded away from zero (respectively from one) for a series of edge probabilities $p(n)$, it must tend to one 
(respectively to zero).

\subsection{The Structure of the Paper}
In section \ref{pre} we give an outline of the proof of theorem \ref{regular} as well as some necessary lemmas. The proof of theorem \ref{regular} is given in section \ref{proofs}. Section \ref{core} contains a sketch of the proof of sharpness of having a non empty $k$-core.


\section{Preliminaries}
\label{pre}
Let us start by giving an outline of the proof.
Let $k\ge3$ be fixed.
We assume, by way of contradiction, that the property of containing a $k$-regular subgraph has a coarse threshold.
By corollary \ref{coarse}, there exist $\alpha$, $p$ and a fixed graph $H$ such 
that for infinitely many $n\in\mathbb{N}$, $\alpha<\mu(p)<1-\alpha$, and the addition of $H$ to $G(n,p)$ 
significantly increases the probability of containing a $k$-regular subgraph.
This will lead to the conclusion that $G(n,p)$ often has many ``almost $k$-regular" subgraphs, namely that if a small number of vertices is chosen randomly, then with probability bounded away from $0$, there will exist a subgraph in which all vertices have degree $k$, apart from the chosen vertices, that have degree $k-1$.
This, in turn, will lead to the conclusion that the addition of a small number of randomly chosen 
edges increases significantly the probability of containing a $k$-regular subgraph.
This will contradict the state of affairs
for properties in $G(n,p)$ that have a coarse threshold, 
where adding a set of randomly chosen edges has a negligible effect whenever the number of added
edges is small compared to the standard deviation of the number of edges in the random graph .

The lemmas in the following subsection will be used for the proof of theorem \ref{regular}.

\subsection{Main Lemmas }
We begin by considering a random graph model that is a small perturbation of $G(n,p)$.
This follows the approach used in \cite{k-color}.
First, a random graph on the vertex set $\set{r+1}{n}$ is chosen as usual, 
by letting each edge appear in the graph independently with probability $p$. 
Next, a matching consisting of $i$ more edges from $\set{1}{r}$ to $\set{r+1}{n}$ is added.
Here $r,i\in\mathbb{N}$ are constants.

\begin{lemma}
\label{Bi}
Let $\alpha,\beta,\eps>0$, $r\in\mathbb{N}$, $S \subseteq\mathbb{N}$ an infinite set, $\mathcal{A}=\mathcal{A}(n)$, $X=X(n)$ graph properties, 
and $p=p(n)$, such that  $p=\Theta(\frac{1}{n})$, $\Pr[X]>\beta$ and $\Pr[\mathcal{A}|X]>\alpha$ for every $n\in S$.

For every  integer $i$ let $B_i$ be the event of graphs on $\set{1}{n}$, in which
\begin{itemize}
\item
 There are exactly $i$ edges between vertices $\set{1}{r}$ and $\set{r+1}{n}$.
\item
 There are no edges connecting two vertices from $\set{1}{r}$.
\item
 There is no vertex in $\set{r+1}{n}$ having more than one neighbour in $\set{1}{r}$.
\end{itemize}
Then there exists $i \ge 0$ such that $\Pr[\mathcal{A}|(B_i\cap X)]>\alpha-\eps$ for infinitely many $n\in\ S$.
\end{lemma}

\begin{proof}[ of lemma \ref{Bi}.]
Denote by $v$ the random variable counting the number of edges between vertices $\set{1}{r}$ and $\set{r+1}{n}$.
$\Ex[v]=p\cdot r(n-r)=O(1)$, as $p=O(\frac{1}{n})$. Hence, there is a $t\in\mathbb{N}$, such that $\Pr[v\ge t]\le\eps\beta/4$. 

Denote by $C$ the event that $v\ge t$, and by $D$ the event in which either there is an edge between vertices from $\set{1}{r}$, or there are vertices from $\set{r+1}{n}$ that are connected by edges to more than one vertex from $\set{1}{r}$. Note that $\Pr[D]\le \binom{r}{2} \cdot p + (n-r) \cdot \binom{r}{2} \cdot (p)^2 = o(1)$. Hence, for large enough $n$, $\Pr[D]\le\eps\beta/4$.

Using the fact that $\Pr[X]>\beta$, We get
\begin{align*}
\alpha \le  \Pr[\mathcal{A}] &
\le \sum_{i=0}^{t-1} \Pr[(\mathcal{A} \cap B_i)|X]  +  \Pr[(\mathcal{A} \cap C)|X]  +  \Pr[(\mathcal{A} \cap D)|X]   \le \\
&\le \sum_{i=0}^{t-1} \Pr[(\mathcal{A} \cap B_i)|X]  +  \Pr[C|X]  +  \Pr[D|X] \le \\
&\le \sum_{i=0}^{t-1} \Pr[(\mathcal{A} \cap B_i)|X]  +  \Pr[C]/\Pr[X]  +  \Pr[D]/\Pr[X] \le \\
&\le \sum_{i=0}^{t-1} \Pr[(\mathcal{A} \cap B_i)|X]  + \eps /2
\end{align*}
for large enough $n \in S$. Therefore
\begin{equation*}
\alpha - \eps/2 \le \sum_{i=0}^{t-1} \Pr[(\mathcal{A} \cap B_i)|X] = \sum_{i=0}^{t-1} \Pr[B_i|X] \cdot \Pr[\mathcal{A} \vert (B_i\cap X)].
\end{equation*}
Hence, as the events $B_i, \  i \in \set{0}{t-1}$ are non intersecting, for every $n\in S$, there is $i_n \in \set{0}{t-1}$ such that
\begin{equation*}
Pr[\mathcal{A} \vert (B_{i_n}\cap X)] \ge\alpha - \eps/2.
\end{equation*}
Let $i$ be such that $i_n=i$ for infinitely many $n\in S$.
\end{proof}

The following lemma will help restricting the discussion of some graph property $\cal{A}$, to the property of
 having a subgraph from $\cal{A}$ with exactly $m$ edges between vertices $\set{1}{r}$ and $\set{r+1}{n}$.
\begin{lemma}
\label{Ae}
Let $\alpha,\eps>0$, $r,M\in\mathbb{N}$, $p=p(n)$, and $\mathcal{A}=\mathcal{A}(n)$ a graph property.

For every integer $m$ denote by $\mathcal{A}_m$ the event that $G(n,p)$ contains as a subgraph, a graph from $\cal{A}$ with exactly $m$ edges between vertices $\set{1}{r}$ and $\set{r+1}{n}$.

 Let $X$ be some event in $G(n,p)$. Assume that $S\subseteq \mathbb{N}$ is an infinite set such that for every $n\in S$:
\begin{itemize}
\item
 $\Pr[\mathcal{A}|X]\ge\alpha$
\item
$\Pr[\mathcal{A}_0|X]<\alpha-\eps$
\item
$\Pr[\mathcal{A}_m|X]=0$, for every $m> M$.
\end{itemize}
Then there exist $\beta>0$, $m\in\mathbb{N}$ such that $\Pr[\mathcal{A}_m|\bar{\mathcal{A}_0}\cap X]\ge \beta$ for an infinite number of $n\in\ S$.
\end{lemma}

\begin{proof}[ of lemma \ref{Ae}.]
For every $n\in S$:
\begin{align*}
\alpha & \le \Pr[\mathcal{A}|X]= \\
&=\Pr[\cup_{m \in \set{0}{M}} \mathcal{A}_m \vert X] =  \\
&=\Pr[\mathcal{A}_0 \cup (\cup_{m \in \set{1}{M}} (\mathcal{A}_m \cap \bar{\mathcal{A}_0})) \vert X] \le    \\
&\le\Pr[\mathcal{A}_0 \vert X] + \sum_{m \in \set{1}{M}} \Pr[\mathcal{A}_m \cap \bar{\mathcal{A}_0} \vert X].
\end{align*}
As $\Pr[\mathcal{A}_0 \vert X]<\alpha-\eps$,
\begin{equation*}
\sum_{m \in \set{1}{M}} \Pr[\mathcal{A}_m \cap \bar{\mathcal{A}_0} \vert X] \ge \eps.
\end{equation*}
Hence, for every $n\in S$ there is $m_n \in \set{1}{M}$, such that $\Pr[\mathcal{A}_{m_n} \cap \bar{\mathcal{A}_0} \vert X] \ge \eps/M \triangleq \beta$. Let $m\in\set{1}{M}$ be such that $m_n=m$ for an infinite number of $n\in\ S$.  As  $\Pr [X \vert Y \cap Z] = \frac{\Pr[X \cap Y \vert Z]}{\Pr[Y \vert Z]} \ge \Pr[X \cap Y \vert Z]$, for every three events $X$, $Y$ and $Z$, we get
\begin{equation*}
\Pr[\mathcal{A}_{m} \vert \bar{\mathcal{A}_0} \cap X] \ge \beta
\end{equation*}
for infinitely many $n\in S$.
\end{proof}

The following lemma will assume the existence of a large number of "almost $k$-regular"
subgraphs , and conclude that the addition of a constant number of edges
has a significant effect on the probability of having a $k$-regular subgraph.
Before stating the lemma, consider the following definition.
\begin{define}
\label{critical}
Let $k\in\mathbb{N}$ and $G$ a graph.  Call a set $L$ of vertices of $G$ \emph{$k$-critical },
if $G$ contains a subgraph in which all vertices have degree $k$, apart from $L$'s vertices that have degree
 $k-1$.
\end{define}

\begin{lemma}
\label{edges}
let $\alpha,\beta>0$, $l,k\in\mathbb{N}$ and $S\subseteq \mathbb{N}$ an infinite set. 
Let $\cal{A}$ be the property of containing a $k$-regular subgraph. 
Let $p=p(n)$ be such that $\Pr[\mathcal{A}]<1-\alpha$.

Assume that for every $n\in S$, when conditioning on $G(n,p)$ having no $k$-regular subgraph, 
a random set of vertices of size $2l$ is $k$-critical with probability at least $\beta$.

Then for infinitely many $n\in S$ the following holds:
Conditioned on $G \in G(n,p)$ not having a $k$-regular subgraph,
the addition of a random set of $l$ edges to $G$ will yield a $k$-regular subgraph
with probability at least $\beta/2$.
\end{lemma}

\begin{proof}[ of lemma \ref{edges}]
Let $T$ be some set of $l$ edges from $K_{n}$ (the complete graph on vertices $1,\ldots,n$). Let $X_T$ be the event that a graph from $G(n,p)$ contains at least one of the edges of $T$.
The probability that a random copy of $G(n,p)$ will contain none of $T$'s edges is $(1-p)^l=1-o(1)$.
Note that in $G(n,p)$: $\Pr[X_T] \ge \Pr[X_T \cap \bar{\mathcal{A}}]=\Pr[X_T | \bar{\mathcal{A}}]\Pr[\bar{\mathcal{A}}] \ge \Pr[X_T | \bar{\mathcal{A}}]\alpha$.
Therefore, as $\Pr[X_T]=o(1)$, $\Pr[X_T| \bar{\mathcal{A}}]=o(1)$.
Also, a random set of $l$ edges of $K_n$, almost surely, contains no pair of intersecting edges.
Hence, for a random graph in $G(n,p)$ with no $k$-regular subgraph, and a random choice of $l$ edges from $K_n$, the probability that the edges will be non intersecting and will not appear in the graph is $1-o(1)$.

It follows from the above and from the conditions of the lemma, that for large enough $n\in S$, conditioning on $G(n,p)$ not having a $k$-regular subgraph, for a random set of $l$ edges from $K_n$, with probability at least $\beta/2$ the following holds: The chosen edges are non intersecting, do not appear in the graph, and their vertices form a $k$-critical set.
The union of the described subgraph and the set of edges is a $k$-regular graph.
 Hence, the addition of a random set of $l$ edges to a graph from $G(n,p)$ having no $k$-regular subgraph, will, with probability at least $\beta/2$, yield a $k$-regular subgraph.
\end{proof}

The following lemma shows that for every graph property, the addition of a small number of random edges
to $G(n,p)$, has a negligible effect on the probability of having the property. 
Its proof can be found in \cite{k-color}.

Let us consider a random graph obtained in the following way. 
First pick a random graph in $G(n,p)$, then choose uniformly at random a set of M pairs of vertices. 
For each pair of vertices, connect the vertices by an edge if they are not already connected. 
Denote this random graph model by $G(n,p,+M)$. 
For a graph property $\cal{A}$, let $\mu^+(p,M)=\mu^+(\mathcal{A},n,p,M)$ be the probability that the resulting graph has property $\cal{A}$.
Note that the number of edges in $G(n,p)$ is a binomial random variable, which is approximately normal with deviation $\sqrt{p\binom{n}{2}}$. Hence it is reasonable to believe that if $M=o(\sqrt{p\binom{n}{2}})$, the two random graph models ($G(n,p)$ and $G(n,p,+M)$) are almost the same. This idea is captured in the following lemma.

\begin{lemma}
\label{addedges}
Let $\cal{A}$ be a monotone graph property.
If $M=o\left(\sqrt{p \binom{n}{2}}\right)$, then $|\mu(p)-\mu^+(p,M)|=o(1)$.
\end{lemma}


\section{Proof of the main theorem}
\label{proofs}

\begin{proof}[ of theorem \ref{regular}]
Let $k \geq 3$ and $\mathcal{A}^*$ be the property of having a $k$ - regular subgraph.
 Assume, for the sake of contradiction, that $\mathcal{A}^*$ does not have a sharp threshold.
Hence, using corollary \ref{coarse} there exist $\alpha^*,\beta>0$, $p=p(n)$, an infinite set $S\subseteq\mathbb{N}$ and a fixed graph $H$ such that for every $n \in S$:
\begin{itemize}
\item
 $\alpha^* <\mu( p)=\Pr[G(n,p(n))\in \mathcal{A}^*] < 1-\alpha^*$.
\item
The probability that a copy of $H$ appears as a subgraph in $G(n,p)$ is at least $\beta$.
\item
$\Pr[\mathcal{A}^*\vert H]>1-\alpha^*/6$ (here we condition on the appearance of some specific copy of $H$).
\end{itemize}

By \cite{k-regular}, $p=\Theta(\frac{1}{n})$.
Hence, using the fact that the probability of a copy of $H$ appearing in $G(n,p)$ is at least $\beta$, $H$ does not contain a $k$-regular subgraph. This is true because, otherwise, some fixed $k$-regular graph would have appeared in $G(n,p)$ with probability at least $\beta$, which is a contradiction to the fact that a threshold function for such a graph is $n^{-2/k}=\omega(\frac{1}{n})$.

Denote the number of vertices in $H$ by $r$ and choose some specific copy of $H$ on the vertices labelled $1$ to $r$.

In lemma \ref{Bi}, take $\cal{A}$ to be the property of the graphs whose union with the specific copy of $H$ described above contains a $k$-regular subgraph, $X=G(n,p)$ and $\alpha=1-\alpha^*/6$,  $\eps=\alpha^*/3$. Let $B_i$ be, as in the lemma, the event in which:
\begin{itemize}
\item
There are exactly $i$ edges between vertices $\set{1}{r}$ and $\set{r+1}{n}$.
\item
There are no edges between two vertices from $\set{1}{r}$.
\item
No vertex from $\set{r+1}{n}$ has more than one neighbour in $\set{1}{r}$.
\end{itemize}
By lemma \ref{Bi}, there exists $i \ge 0$ such that $\Pr[\mathcal{A}|B_i]\ge 1-\alpha^*/2$, for every $n\in S_1$, where $S_1\subseteq S$ is an infinite set.
Note that
\begin{equation*}
\Pr[\mathcal{A}|B_0]=\Pr[G(n-r,p(n))\in \mathcal{A}^*]\le\Pr[G(n,p(n))\in \mathcal{A}^*]\le 1-\alpha^*,
\end{equation*}
as  the event of $G(n,p)$ having a $k$-regular subgraph contains the event of $G(n,p)$ having a $k$-regular subgraph on vertices $\set{r+1}{n}$. Hence $i \ne 0$.

In lemma \ref{Ae}, take $\cal{A}$ as before, to be the graph property of all graphs whose union with the specific copy of $H$ is in $\mathcal{A}^*$, take $X=B_i$, $\alpha=1-\alpha^*/2$, $\eps=\alpha^*/4$.
Let $\mathcal{A}_m$ be, as in the lemma, the property of graphs having a subgraph from $\cal{A}$ with exactly $m$ edges between vertices $\set{1}{r}$ and $\set{r+1}{n}$.
Note that
\begin{itemize}
\item
$\Pr[\mathcal{A}|B_i]\ge 1-\alpha^*/2=\alpha$, as stated above as a conclusion to lemma \ref{Bi}.
\item
$\Pr[\mathcal{A}_0|B_i]=\Pr[ G(n-r,p(n)) \in \mathcal{A}^*] \le \Pr[ G(n,p(n)) \in \mathcal{A}^*]\le 1-\alpha^*<1-\frac{3}{4}\alpha^*=\alpha-\eps$.
\item
$\Pr[\mathcal{A}_m|B_i]=0$, as for every $m>i$, when conditioning on $B_i$, there are exactly $i$ edges between vertices $\set{1}{r}$ and $\set{r+1}{n}$, so there can be no subgraph having more than $i$ edges between these sets of vertices.
\end{itemize}
By lemma \ref{Ae}, there exist  $m \in \mathbb{N}$ and $\gamma>0$ such that $\Pr[\mathcal{A}_m|\bar{\mathcal{A}_0}\cap B_i]\ge \gamma$,
 for every $n\in S_2\subseteq S_1$, where $S_2$ is an infinite set. This leads to the following conclusion (recall that by definition \ref{critical}, a set is $k$-critical with respect to a certain graph, if there is a subgraph such that all its vertices are of degree $k$, except for the set's vertices which are of degree $k-1$).

\begin{cor}
\label{clear}
For every $n\in S_2$, conditioning on $G(n-r,p)$ not having a $k$-regular subgraph, 
a random set of $m$ vertices is $k$-critical with probability at least $\delta$, for some $\delta>0$.
\end{cor}

\begin{proof}
Let us show first that for every $n\in S_2$, conditioning on $G(n-r,p)$ not having a $k$-regular subgraph, 
a random set of $i$ vertices has a $k$-critical subset of size $m$ with probability at least $\gamma$. \\
Look at the restriction of $G(n,p)$ to the vertex set $\set{r+1}{n}$.
Using the fact that $\Pr[\mathcal{A}_m|\bar{\mathcal{A}_0}\cap B_i]>\gamma$ the conclusion can be reached by letting the random set of size $i$ correspond to the neighbours of $\set{1}{r}$, and by choosing a subset of size $m$ to be the set of neighbours of $\set{1}{r}$  in a $k$-regular subgraph with $m$ edges between $\set{1}{r}$ and $\set{r+1}{n}$ whenever such a graph exists.

Stating this conclusion differently, 
for a random choice of a graph and a set as described above, the following holds: A random subset of 
size $m$ of the chosen set of $i$ vertices, will be $k$-critical with probability at least $\delta\triangleq\gamma/\binom{i}{m}$.

Hence, for every $n\in S_2$, with probability at least $\delta$, the following holds: When conditioning on $G(n-r,p)$ not having a $k$-regular subgraph, a random set of $m$ vertices is $k$-critical.
\end{proof}

Assume, first, that $m$ is even. 

By lemma \ref{edges}, and the fact that $\Pr[G(n-r,p)\in \mathcal{A}^*]\le\Pr[G(n,p)\in \mathcal{A}^*]\le 1-\alpha^*$, there exists an infinite set $S_3 \subseteq S_2$ such that for every $n\in S_3$, the following holds: When conditioning on $G(n-r,p)$ not having a $k$-regular subgraph, with probability at least $\delta/2$, the addition of $m/2$ edges to the graph will yield a graph containing a $k$-regular subgraph.

Recall the notations of lemma \ref{addedges}:
\begin{itemize}
\item
$\mu(\mathcal{A},n,p)\triangleq\Pr[G(n,p)\in \mathcal{A}]$.
\item
$\mu^+(\mathcal{A},n,p,M)$ is the probability that the union of a random graph from $G(n,p)$ and a random set of $M$ edges has property $\mathcal{A}$.
\end{itemize}
Using these notations, we get:

\begin{align*}
&\mu^+(\mathcal{A}^*, n-r, p, m/2) \ge \\
&\ge \mu(\mathcal{A}^*,n-r,p)+(1-\mu(\mathcal{A}^*,n-r,p))\cdot\delta/2 \ge \\
&\ge \mu(\mathcal{A}^*,n-r,p)+\alpha^*\delta/2 .
\end{align*}
Therefore
\begin{equation*}
|\mu^+(\mathcal{A}^*, n-r, p, m/2)-\mu(\mathcal{A}^*,n-r,p)| \ge \alpha^*\delta/2.
\end{equation*}
This is a contradiction to lemma \ref{addedges}.

Assume now that $m$ is odd.
Then $k$ is necessarily odd as well, because, otherwise, there exists a graph whose vertices are all of degree $k$, apart from $m$ vertices that are of degree $k-1$. This leads to a contradiction since the sum of degrees of vertices of the described graph is $k \cdot x + (k-1)m$ (where $x$ is the number of vertices of degree $k$), which is an odd number if $k$ is even.

Recall that we know that for every $n\in S_2$, conditional on $G(n-r,p)$ having no $k$-regular subgraph, a random set of vertices of size $m$ is $k$-critical with probability at least $\delta$.
By the symmetry, for every $n\in S_2$, with probability at least $\delta$, the set $\set{1}{m}$ is $k$-critical with respect to a random graph from $G(n-r,p)$ having no $k$-regular subgraphs.

\begin{claim}
\label{last}
There exists $\zeta>0$ such that when conditioning on the restriction of $G(n-r,p)$ to $\set{m+1}{n-r}$ having no $k$-regular subgraph, the set $\set{1}{m}$ is $k$-critical with probability at least $\zeta$.
\end{claim}

\begin{proof}[ of claim \ref{last}]
Let
\begin{itemize}
\item
$\mathcal{A}'$ be the property of graphs for which $\set{1}{m}$ is $k$-critical.
\item
$B$ be the event that there is no $k$-regular subgraph on vertices $\set{1}{n-r}$.
\item
$C$ be the event that there is no $k$-regular subgraph on vertices $\set{m+1}{n-r}$.
\end{itemize}
Note that:
\begin{align*}
\Pr[B]&\ge \Pr[{\text{ \{vertices } 1, \ldots m \text{ are of degree } 0\text{\} }}\cap C]= \\
&=\Pr[\text{vertices } 1, \ldots m \text{ are of degree } 0]\cdot\Pr[C] 
\end{align*}
As the events are independent.

Therefore, there exists some $\zeta>0$ such that
\begin{equation*}
\frac{\Pr[B]}{\Pr[C]}\ge\Pr[\text{vertices } 1, \ldots m \text{ are of degree } 0]=(1-p)^{(n-r-m)m+\binom{m}{2}}\ge \zeta
\end{equation*}
as $p=\Theta(1/n)$. 

Hence for every $n\in S_2$:
\begin{equation*}
\Pr[\mathcal{A}'|C] = 
\frac{\Pr[\mathcal{A}' \cap C]}{\Pr[C]} \ge \frac{\Pr[\mathcal{A}' \cap B]}{\Pr[C]} = 
\frac{\Pr[\mathcal{A}' \cap B]}{\Pr[B]}\cdot\frac{\Pr[B]}{\Pr[C]} =
\Pr[\mathcal{A}'|B]\cdot\frac{\Pr[B]}{\Pr[C]} \ge  
\delta \zeta.
\end{equation*}

\end{proof}

In lemma \ref{Bi}, take $\mathcal{A}$ to be the property of graphs in which $\set{1}{m}$ is $k$-critical, and let $X$ be the event of there being no $k$-regular subgraph with vertices from $\set{m+1}{n-r}$. Note that $\Pr[X]=\Pr[G(n-r-e,p)\notin\mathcal{A}^*]\ge\Pr[G(n,p)\notin\mathcal{A}^*]\ge\alpha^*$. 
By the lemma, there exists $j \ge 0$, and an infinite set $S_3\subseteq S_2$ such that for every $n\in S_3$, 
\small
\begin{align*}
\zeta/2\le&\Pr[\mathcal{A}|B_j\cap X]= \\
=&\Pr[(\set{1}{m}\text{ is }k\text{-critical})|B_j\cap(\text{no }k\text{-regular subgraph with vertices from }\set{m+1}{n-r})].
\end{align*}
\normalsize
As before (in the proof to corollary \ref{clear}), by looking at the restriction of the graph to the vertices $\set{m+1}{n-r}$ and at the neighbours of $\set{1}{m}$ in a subgraph in which the only vertices not of degree $k$ are vertices $\set{1}{m}$, which are of degree $k-1$, we get:

\begin{cor}
For every $n\in S_3$, when conditioning on $G(n-r-m,p)$ not having a $k$-regular subgraph, a random set of vertices of size $m'\triangleq m(k-1)$ is $k$-critical, with probability at least $\zeta/2$.
\end{cor}

Continuing in the same way as in the case where $m$ was even 
(note that $m'$ is even as $k$ is odd), we will get, again, a contradiction to lemma \ref{addedges}.
\end{proof}


\section{Sharpness of the property of containing a non empty $k$-core }
\label{core}

The proof of the sharpness of having a non empty $k$-core is very similar to the proof of the sharpness of containing a $k$-regular subgraph, therefore only a sketch of the proof will be given here.
Let $k \geq 3$ and $\mathcal{A}^*$ be the property of having a non empty $k$-core.
 Assume, for the sake of contradiction, that $\mathcal{A}^*$ has a coarse threshold.
Hence, using corollary \ref{coarse}, there exist $\alpha^*,\beta>0$, $p=p(n)$, an infinite set $S\subseteq\mathbb{N}$ and a fixed graph $H$ such that for every $n \in S$:
\begin{itemize}
\item
 $\alpha^* <\mu( p) < 1-\alpha^*$.
\item
The probability that a copy of $H$ appears as a subgraph in $G(n,p)$ is at least $\beta$.
\item
$\Pr[\mathcal{A}^*\vert H]>1-\eps$.
\end{itemize}

Note that a graph on $n$ vertices with more than $(k-1)n$ edges has a non trivial $k$-core (because the $k$-core of a graph could be found by erasing all vertices with degree lower than $k$ together with their edges, until no such vertices exist. In case the $k$-core is empty, all vertices will be removed in this process, yet no more than $(k-1)n$ edges can be removed). Hence $p=O(\frac{1}{n})$.
Therefore, using the fact that the probability of a copy of $H$ appearing in $G(n,p)$ is at least $\beta$, $H$ has an empty $k$-core. This is true because, otherwise, some fixed graph with minimal degree at least $k$ would have appeared in $G(n,p)$ with probability at least $\beta$. This is a contradiction to the fact the threshold function for such a graph is at least $n^{-2/k}=\omega(\frac{1}{n})$.

Let $r$ be the number of vertices of $H$.
Denote by $\cal{A}$  the property of the graphs whose union with a specific copy of $H$ on vertices $\set{1}{r}$ lies in $\mathcal{A}^*$.
It can be concluded, as in the proof of \ref{regular}, using lemmas \ref{Bi} and \ref{Ae}, that there are $m,i\in\mathbb{N}$ and $\gamma>0$  such that $\Pr[\mathcal{A}_m|\bar{\mathcal{A}_0}\cap B_i]\ >\gamma$ for every $n\in S_1\subseteq S$, where $S_1$ is an infinite set.

Consider the following definition.
\begin{define}
Let $k\in\mathbb{N}$, and let $G$ be a graph. A set $L$ of vertices of $G$ is called \emph{$k$-critical*} if there exists a subgraph of $G$ in which $L$'s vertices have degree at least $k-1$ and all other vertices have degree at least $k$.
\end{define}
Using this definition, the above leads to the conclusion (similarly to the proof of corollary \ref{clear}) that there exists $\delta>0$ such that for every $n\in S_1$ the following holds: When conditioning on $G(n-r,p)$ having an empty $k$-core, with probability at least $\delta$ a random set of $m$ vertices is $k$-critical*.

It follows easily that there exists  $\zeta>0$ such that for infinitely many $n\in S_1$, when conditioning on $G(n-r,p)$ having an empty $k$-core, a random set of size $2m$ is $k$-critical* with probability at least $\zeta$.

Using a variation of lemma \ref{edges}, it can be concluded that there exists $\zeta'>0$, such that for infinitely many $n\in S_1$, the following holds:
When conditioning on $G(n-r,p)$ having an empty $k$-core, the addition of $m$ edges to the graph will yield a non empty $k$-core with probability at least $\zeta'$.
With the same argument as in the proof of theorem \ref{regular}, this will lead to a contradiction to lemma \ref{addedges}.

\section{ Acknowledgements}
I would like to thank Ehud Friedgut for introducing me to the 
subject of sharp thresholds in random graphs in general, and to the problems treated in this paper in particular. 
I would also like to thank him for supporting and guiding me in the process of finding the proof and writing it.

\bibliography{k-regular}
\bibliographystyle{abbrv}

\end{document}